\documentclass{amsart}
\usepackage[latin1]{inputenc}
\usepackage{amsmath,amsfonts, amssymb}

\newcommand{\BB}{{\mathbb{B}}}
\newcommand{\RR}{{\mathbb R}}
\newcommand{\thetahat}{{\hat{\theta}}}
\newcommand{\rhat}{{\hat{r}}}
\newcommand{\ainf}{{a_\infty}}
\newcommand{\ainfsq}{{a_\infty^2}}

\newcommand{\Rp}{{R^{\prime}}}
\newcommand{\Rpp}{{R^{\prime\prime}}}
\newcommand{\ubar}{{\overline{u}}}
\newcommand{\Yl}{{Y_l}}

\DeclareMathOperator{\grad}{grad}
\DeclareMathOperator{\mydiv}{div}
\DeclareMathOperator{\sign}{sign}

\newcommand{\Ls}{{\Delta_S}}
\newcommand{\Ns}{{\nabla_{\!\!S}}}

\newcommand{\sgrad}{{\grad_{\partial\BB}}}
\newcommand{\sdiv}{{\mydiv_{\partial\BB}}}
\newcommand{\surfintB}{{\int_{\partial\BB}}}

\newtheorem{thm}{Theorem}
\newtheorem{lemma}[thm]{Lemma}
\newtheorem{prop}[thm]{Proposition}
\theoremstyle{remark}
\newtheorem*{rmk}{Remark}

\title[Circular free plates]{Vibrational modes of circular free plates under tension}
\author{L. M. Chasman}
\address{Box K-67, Knox College, 2 East South Street, Galesburg, IL 61401 U.S.A.} \email{lchasman$@$knox.edu}
\date{\today}


\begin{document}
\begin{abstract}
The vibrational frequencies of a plate under tension are given by the eigenvalues $\omega$ of the equation $\Delta^2u-\tau\Delta u=\omega u$. This paper determines the eigenfunctions and eigenvalues of this bi-Laplace problem on the ball under natural (free) boundary conditions. In particular, the fundamental modes --- the eigenfunctions of the lowest nonzero eigenvalue --- are identified and found to have simple angular dependence.
\end{abstract}
\maketitle

\section{Introduction}
The importance of the disk for physical isoperimetric inequalities has been recognized since 1877, when Lord Rayleigh \cite{RToS} conjectured that the first eigenvalue of the Dirichlet Laplacian on a planar domain (the fundamental tone of a fixed membrane) is bounded below by the first eigenvalue of the disk of the same area. This conjecture was later proved by Faber \cite{Faber53} and Krahn \cite{krahn25,krahn26}. The disk is also the extremal case in the Kornhauser--Stakgold upper bound for the Neumann eigenvalue (free membrane), proven by Szeg\H o \cite{S50,Serrata} and Weinberger \cite{W56} in the 1950's.

The disk further gives extremal cases for vibrating clamped plates. Plate vibrations are governed by the bi-Laplacian operator just as the Laplacian governs vibrations of membranes. The collective work of Szeg\H o \cite{S50}, Talenti \cite{T81}, and Nadirashvili \cite{N92,N95} (and Ashbaugh, Benguria, and Laugesen \cite{AB95, AL96, ABL97} in higher dimensions) established a sharp lower bound for the fundamental tone of a clamped plate --- that is, the first eigenvalue of the bi-Laplacian with boundary conditions $u=0$ and $\partial u/\partial n=0$. As in the Faber-Krahn inequality, the disk provides the lower bound on the first eigenvalue.

In this paper, we consider the free (unconstrained) plate problem, with the goal of identifying the eigenfunctions and eigenvalues of the ball, and in particular the fundamental mode. We consider the more general situation of a free plate under tension and factor the eigenvalue equation in order to show the ball eigenfunctions can be written in terms of Bessel and modified Bessel functions and spherical harmonics. We identify the fundamental tone using the plate Rayleigh quotient and the specific form of the free boundary conditions. 

Free boundary conditions are more complicated than clamped ones. This makes determination of the free circular plate eigenfunctions somewhat more difficult than for the clamped circular plate \cite[Chapter 2]{leissa}, which has only been treated in the zero tension case as far as we know (see, e.g., the survey paper \cite{ABL97}). Other boundary conditions exist, such as the hinged plate investigated by Nazarov and Sweers \cite{NS07}, and the simply supported plate. 

The results of this paper identifying the fundamental mode of the unit ball are essential to the solution of the free plate isoperimetric problem, which we present in \cite{chasman}. That result states that the fundamental tone of a free plate under tension is sharply bounded above by the fundamental tone of the ball of same volume, which is analogous to the Szg\"o-Weinberger result for free membranes.

The circular plate is also significant because it is a rare example where we are able to find the form of its eigenfunctions. The eigenvalue equation for the free plate is not separable in rectangular coordinates due to the cross-term $u_{xxyy}$ in $\Delta\Delta u$, and so we cannot easily find the eigenfunctions for the rectangular plate.

Problems with the bi-Laplacian tend to be more difficult than their second-order counterparts, because the theory of the bi-Laplacian is not nearly so well developed.  For example, the maximum principle fails for the bi-Laplacian, and solvability of the biharmonic equation in Lipschitz domains with Neumann boundary conditions was established only a few years ago by Verchota \cite{verchota}.

Other notable mathematical work on plates includes Kawohl, Levine, and Velte \cite{KLV}, who investigated the sums of low eigenvalues for the clamped plate under tension and compression, and Payne \cite{P58}, who considered both vibrating and buckling free and clamped plates and established inequalities bounding plate eigenvalues by their (free or fixed) membrane counterparts.

\section{Formulating the problem}
In this section we present the mathematical formulation of the free plate problem, summarizing known facts about existence of the spectrum, boundary conditions, and regularity of solutions as proved in \cite{chasman}. We consider dimensions $d\geq2$; the elementary one-dimensional case (free rod) is discussed in \cite[Chapter 7]{cthesis}. In this paper we consider only regions $\Omega=\BB(R)$, balls of radius $R$; we write $\BB=\BB(1)$.

The eigenvalues of the free plate under tension with shape $\BB(R)$ are characterized by the Rayleigh quotient
\begin{equation}
Q[u] = \frac{\int_{\BB(R)} |D^2 u|^2 + \tau |D u|^2\,dx}{\int_{\BB(R)} |u|^2\,dx}. \label{RQ}
\end{equation}
Here $|D^2u|=(\sum_{jk}u_{x_jx_k}^2)^{1/2}$ is the Hilbert-Schmidt norm of the Hessian matrix $D^2u$, and $Du$ denotes the gradient vector. The parameter $\tau>0$ has the physical meaning of the ratio of tension to flexural rigidity. Note in particular that we have the Hessian term $|D^2u|^2$, involving all second derivatives, rather than only those of the Laplacian term $(\Delta u)^2$; see \cite[p. 228]{weinstock}.

From this quotient, we obtain the Euler--Lagrange partial differential equation
\begin{equation}
\Delta \Delta u - \tau \Delta u = \omega u, \label{maineq}
\end{equation}
where $\omega$ is the eigenvalue, together with the natural (\emph{i.e.}, unconstrained or ``free'') boundary conditions on $\partial\BB(R)$ \cite[Proposition 7]{chasman}:
\begin{align}
&Mu := u_{rr} = 0 &\text{at $r=R$,}\label{BCb1}\\
&Vu := \tau u_r-\frac{1}{r^2}\Delta_S\left(u_r-\frac{u}{r}\right)-(\Delta u)_r = 0&\text{at $r=R$.}\label{BCb2}
\end{align}
Here $\Delta_S$ is the angular part of the Laplacian, which in two dimensions is simply $\partial^2/\partial\theta^2$.

The spectrum consists of nonnegative eigenvalues of finite multiplicity
\[
 0=\omega_0<\omega_1\leq \omega_2\leq \dots \to\infty.
\]
The corresponding eigenfunctions are smooth on $\overline{\BB(R)}$. The eigenfunction corresponding to $\omega_0=0$ is the constant function.

We call $\omega_1$ the \emph{fundamental tone} and a corresponding eigenfunction $u_1$ a \emph{fundamental mode}. 

The radial part of the eigenfunctions for the ball will be obtained in terms of ultraspherical Bessel functions. In order to state the main results of this paper, we will need to define these functions. The ultraspherical Bessel functions $j_l(z)$ of the first kind are defined in terms of the Bessel functions of the first kind, $J_\nu(z)$, as follows:
\[
j_{l}(z) = z^{-s}J_{s+l}(z)
\]
where the parameter
\[
 s = \frac{d-2}{2}
\]
depends on the dimension $d$. This notation suppresses the dependence of the $j_l$ functions on the dimension $d$, which causes no problems since the dimension $d\geq 2$ is fixed. Ultraspherical modified Bessel functions $i_{l}(z)$ of the first kind are defined analogously:
\[
i_{l}(z) = z^{-s}I_{s+l}(z),
\]
where $I_\nu$ is the modified Bessel function of the first kind.

\section{Main Results}
The aim of this paper is to find the eigenfunctions of the free plate problem on the ball and identify the fundamental mode. We will only consider $\tau>0$, corresponding to the free plate under tension. 

The first result identifies the form of the eigenfunctions of the ball of radius $R>0$.
\begin{prop}\label{eigfc} (Eigenfunctions in spherical coordinates) Let $\tau>0$ and $\omega$ be any positive eigenvalue of the free ball $\BB(R)$; that is, $\omega$ is an eigenvalue of $\Delta\Delta u-\tau\Delta u =\omega u$ under boundary conditions \eqref{BCb1} and \eqref{BCb2}. Then the corresponding eigenfunctions can be written in the form $R_l(r)\Yl(\thetahat)$, where $\Yl$ is a spherical harmonic of some integer order $l$ and $R_l$ is a linear combination of ultraspherical Bessel and modified Bessel functions,
\[
R_l(r) = j_l(ar/R)+\gamma i_l(br/R).
\]
Here the positive numbers $a$ and $b$ depend on $\tau$ and $\omega$ by $b^2-a^2=R^2\tau$ and $a^2b^2=R^4\omega$, and $\gamma$ is a real constant given by
\[
 \gamma=\frac{-a^2j_l^{\prime\prime}(a)}{b^2i_l^{\prime\prime}(b)}.
\]
\end{prop}

This result gives us the forms of eigenfunctions given an eigenvalue $\omega$. However, we do not know the eigenvalues, and instead wish to find them. The next result, a consequence of the proof of Proposition~\ref{eigfc}, allows us to compute the eigenvalues of the ball:
\begin{prop}\label{eigval} (Eigenvalues) Fix $\tau>0$. Then a real number $\omega$ is an eigenvalue of the free plate of shape $\BB(R)$ if and only if $\omega = 0$ or there exists an integer $l\geq 0$ and positive real constants $a$, $b$ such that $a^2b^2=R^4\omega$, $b^2-a^2=R^2\tau$, and
\begin{align*}
 W_l(a)&:=a^2j_l^{\prime\prime}(a)\left(-a^2bi_l^\prime(b)+l(l+d-2)\Big(bi_l^\prime(b)-i_l(b)\Big)\right)\\
&\qquad-b^2i_l^{\prime\prime}(b)\left(ab^2j_l^\prime(a)+l(l+d-2)\Big(aj_l^\prime(a)-j_l(a)\Big)\right)=0.
\end{align*}
\end{prop}

\begin{rmk} This proposition gives us a way to calculate eigenvalues numerically. Given $\tau$ and a nonnegative integer $l$, the roots of $W_l(a)$ will determine eigenvalues by the relation $\omega=a^2(a^2+\tau)$.
\end{rmk}

We are particularly interested in identifying the fundamental tone and mode of the ball, because the proof of the free plate isoperimetric inequality in \cite{chasman} uses trial functions constructed from that fundamental mode. Thus the following theorem is essential to the proof in \cite{chasman}.

\begin{thm} \label{thm2} (Fundamental mode of the ball) For $\tau>0$, the fundamental modes of the  ball $\BB(R)$ can be written as linear combinations of
\[
u_1(r,\thetahat)=\Big(j_1(ar/R)+\gamma i_1(br/R)\Big)Y_1(\thetahat),
\]
with $a$, $b$, $\gamma$ real constants, with $a$ and $b$ positive and depending on $\tau$ and $\omega_1$ as follows: $b^2-a^2=R^2\tau$ and $a^2b^2=R^4\omega_1$, and $\gamma$ given by
\[
 \gamma=\frac{-a^2j_l^{\prime\prime}(a)}{b^2i_l^{\prime\prime}(b)}.
\]
In particular, in dimension 2, the fundamental modes are linear combinations of
\[
 u_1=\Big(J_1(ar/R)+\gamma I_1(br/R)\Big)\begin{cases}\sin(\theta)\\ \cos(\theta)\end{cases}.
\]
\end{thm}

\begin{rmk} When $\tau=0$, linear functions are eigenfunctions with eigenvalue zero, and so the zero eigenvalue is $d+1$-fold degenerate; the fundamental modes will then involve higher-order spherical harmonics.

This model considered with negative values of the tension parameter $\tau$ corresponds to the free plate under compression. In this case, the Rayleigh quotient $Q$ yields both positive eigenvalues (corresponding to vibrational modes) and negative eigenvalues (corresponding to buckling modes). The forms of the eigenfunctions in the buckling modes differ from those of Proposition~\ref{eigfc}. Furthermore, for $\tau<0$ the fundamental modes can involve higher-order spherical harmonics.
\end{rmk}

\section{Ultraspherical Bessel Functions of the First Kind}
We must examine properties of $d$-dimensional ultraspherical Bessel functions, for they provide the radial part of the eigenfunctions of the ball. For more information on Bessel functions, see \cite[p.358-389]{AShandbook}. For more information on spherical and ultraspherical Bessel functions, see  \cite[p.437-455]{AShandbook} ($d=3$ only) and \cite{LS94} (all $d\geq2$).

The Bessel $j_l(z)$ functions defined previously solve the ultraspherical Bessel equation,
\begin{equation}
z^2w''+(d-1)zw'+\Big(z^2-l(l+d-2)\Big)w = 0.\label{besseleqn}
\end{equation}
The modified Bessel $j_l(z)$ functions solve the modified ultraspherical Bessel equation,
\begin{equation}
z^2w''+(d-1)zw'-\Big(z^2+l(l+d-2)\Big)w = 0. \label{modbesseleqn}
\end{equation}

\subsection*{Recurrence Relations and power series}
The Bessel functions $J_\nu$ and $I_\nu$ have a number of useful recurrence relations; see, for example, \cite[p. 361, 376]{AShandbook}. The ultraspherical Bessel functions have similar recurrence relations, all of which follow from the definition and application of the corresponding ordinary Bessel recurrence relations.
\begin{align}
\frac{d-2+2l}{z}j_l(z) &= j_{l-1}(z)+j_{l+1}(z)\label{j1}\\
j_l^\prime(z) &= \frac{l}{z}j_l(z)-j_{l+1}(z)\label{j2}\\
&=j_{l-1}(z)-\frac{l+d-2}{z}j_l(z)\label{j3}\\
\frac{d-2+2l}{z}i_l(z) &= i_{l-1}(z)-i_{l+1}(z)\label{i1}\\
i_l^\prime(z) &= \frac{l}{z}i_l(z)+i_{l+1}(z)\label{i2}
\end{align}
Note that if we take $d=2$, each of these simplifies to the corresponding relation for Bessel functions.

We also have recurrence relations for the second derivatives:
\begin{align}
j_l^{\prime\prime}(z) &= \left(\frac{l^2-l}{z^2}-1\right)j_l(z)+\frac{d-1}{z}j_{l+1}(z)\label{j4}\\
i_l^{\prime\prime}(z) &= \left(\frac{l^2-l}{z^2}+1\right)i_l(z)-\frac{d-1}{z}i_{l+1}(z)\label{i4}.
\end{align}
Again, when $d=2$ each recurrence relation simplifies to its two-dimensional analog.

We may also write a power series for the ultraspherical Bessel functions $j_l(z)$ and $i_l(z)$ using the series for the corresponding $J_{s+l}$ and $I_{s+l}$:
\begin{align}
 j_l(z) &= \sum_{k=0}^\infty\frac{(-1)^k\,2^{1-d/2}}{k!\,\Gamma(k+\frac{d}{2}+l)}\left(\frac{z}{2}\right)^{2k+l}\label{jseries}\\
i_l(z) &=\sum_{k=0}^\infty\frac{2^{1-d/2}}{k!\,\Gamma(k+\frac{d}{2}+l)}\left(\frac{z}{2}\right)^{2k+l} .\label{iseries}
\end{align}
By examining the power series \eqref{iseries}, it is immediate that $i_l(z)$ and its derivatives are all positive on $(0,\infty)$. Since the terms of the power series for $j_l$ and $i_l$ are the same up to a sign, we also have that the derivatives of $j_l$ are dominated by those of $i_l$:
\begin{equation}
\Big|j_l^{(m)}(z)\Big| \leq i_l^{(m)}(z)\qquad \text{for $m\geq0$, $z\geq 0$,}
\end{equation}
with equality only at $z=0$.

\subsection*{Other needed facts}
To identify the fundamental tone of the circular plate, and to prove the isoperimetric inequality for the free plate in \cite{chasman}, we will need several facts about Bessel functions and their derivatives. We begin with a result on the zeros of the $j_l^\prime(z)$.

\begin{prop}[L. Lorch and P. Szego, \cite{LS94}]\label{propLS}
Let $p_{l,k}$ denote the $k$th positive zero of $j_l^\prime(z)$. Then for $d\geq3$ and $l\geq 1$,
\[
 \frac{l(d+2l)(d+2l+2)}{d+4l+2}<\left(p_{l,1}\right)^2<l(d+2l).
\]
In particular, for $p_{1,1}$ the first zero of $j_1^\prime$, we deduce
\[
d< p_{1,1}^2< d+2.
\]
This inequality holds for all $d\geq2$.
\end{prop}
Recall $s=(d-2)/2$. This first lemma below gives us information on the signs of the Bessel $j_l$ functions, and will be used in \cite{chasman} to help prove the free plate isoperimetric inequality.

\begin{lemma} \label{fact1} The functions $j_l$ and $J_{s+l}$ have the same sign. In particular, for $l= 1,\dots,5$ and any $d\geq 2$, we have $j_l(z)>0$ for $z\leq p_{1,1}$.
\end{lemma}
\begin{proof} The first statement is immediate from the definition of the ultraspherical Bessel functions. For the second statement, we appeal to established facts of Bessel functions.
We write $j_{l,1}$ for the first nontrivial zero of the Bessel function $J_l(z)$. It is well known that $J_l(z)$ is positive on $(0,j_{l,1})$ and the zeros $j_{l,1}$ are increasing in $l$ for $l\geq 1$. Because $J_1(z)=0$ at $z=0$ and $j_{1,1}$ with no zeros between, we have the same for $j_1(z)$ and thus the first root of $j_1^\prime(z)$, $p_{1,1}$, lies between $0$ and $j_{1,1}$. Therefore for any $d\geq 2$ and any $l\geq1$, we have $J_l(z)>0$ and hence $j_l(z)>0$ on $(0,p_{1,1}]$.
\end{proof}

\begin{lemma} \label{fact1.5} We have $j_1^\prime>0$ on $(0,p_{1,1})$.
 \end{lemma}
\begin{proof}
 This follows from the observation that $j_1(z)>0$ on $(0,j_{1,1})$ with $j_1(0)=0$ and the definition of $p_{1,1}$ as the first zero of $j_1^\prime(z)$.
\end{proof}

\begin{lemma}\label{fact2} We have $j_2'>0$ on $(0,p_{1,1}]$.
\end{lemma}
\begin{proof} Let $p_{2,1}$ denote the first zero of $j_2'$. By Proposition~\ref{propLS}, $p_{1,1}^2<d+2$ and
\[
p_{2,1}^2 >\frac{2(d+4)(d+6)}{d+10}.
\]
Then $p_{2,1}^2-p_{1,1}^2>(d^2+8d+28)/(10+d)>0$, so that $j_2^\prime>0$ on $(0,p_{1,1}]$.
\end{proof}

\begin{lemma}\label{fact3} We have $j_1'' < 0$ on $(0,p_{1,1}]$.
\end{lemma}
\begin{proof}
We see that
\begin{align*}
j_1^{\prime\prime}(z) &= \frac{d-1}{z}j_2(z)-j_1(z) &&\text{by \eqref{j4}}\\
&=-\frac{1}{z}j_2(z)-j_2^\prime(z)&&\text{by \eqref{j3} with $l=2$.}
\end{align*}
Since both $j_2$ and $j_2^\prime$ are positive on $(0,p_{1,1}]$ by the previous lemmas, we obtain $j_1^{\prime\prime}>0$ on that same interval.
\end{proof}

The next two lemmas are needed for the proof of the free plate isoperimetric inequality in \cite{chasman} but are not used in this paper. We include them here since both follow from properties of Bessel functions established in this paper.

\begin{lemma}\label{fact4} We have $j_1^{(4)}>0$ on $(0,p_{1,1}]$.
\end{lemma}
\begin{proof}
We have by \eqref{j4} that
\[
 j_1^{\prime\prime}(z)=-j_1(z)+\frac{d-1}{z}j_2(z),
\]
and so
\begin{align}
 j_1^{(4)}&=-j_1^{\prime\prime}(z)+\frac{d-1}{z}j_2^{\prime\prime}(z)-\frac{2(d-1)}{z^2}j_2^\prime(z)+\frac{2(d-1)}{z^3}j_2(z) \nonumber \\
&=j_1(z)-\frac{2(d-1)}{z}j_2(z)+\frac{d^2-1}{z^2}j_3(z) \label{fact4eq1}
\end{align}
by \eqref{j4} with $l=1$ and $l=2$, and \eqref{j2} with $l=2$.
When $d=2$, this becomes
\begin{equation}
  j_1^{(4)}(z)=\left(1-\frac{3}{z^2}\right)j_1(z)+\left(\frac{12}{z^3}-\frac{2}{z}\right)j_2(z) \label{d2case}
\end{equation}
by \eqref{j1} with $l=2$. For any $d$, \eqref{fact4eq1} gives us
\begin{align}
 j_1^{(4)}(z)&=\frac{4-d}{z}j_2(z)+\left(\frac{d^2-1}{z^2}-1\right)j_3(z)
&\quad\text{by \eqref{j1} with $l=2$}\label{dsmall}\\
&=\left(\frac{15}{z^2}-1\right)j_3(z)+\frac{d-4}{z}j_4(z) &\quad\text{by \eqref{j1} with $l=3$}\label{dbigger1}\\
&=\left(\frac{15(d+6)}{z^3}-\frac{10}{z}\right)j_4(z)+\left(1-\frac{15}{z^2}\right)j_5(z)
&\quad\text{by \eqref{j1} with $l=4$}\label{dbigger2}.
\end{align}

When $d=2$, then the first term of \eqref{dsmall} is nonnegative on $(0,p_{1,1}]$ by Lemma~\ref{fact2}. The function $j_3$ is positive on $(0,p_{1,1}]$ by Lemma~\ref{fact1}; note that since $d=2$, we have $j_3(z)=J_3(z)$. Thus we have $j_1^{(4)}(z)>0$ when $z\in(0,\sqrt{3}]\cap(0,p_{1,1}]$. However, $p_{1,1}>\sqrt{3}$, so we have only established positivity on $(0,\sqrt{3}]$.

To establish positivity on $(\sqrt{3},p_{1,1}]$ we turn to \eqref{d2case}. The first term is certainly positive on $(\sqrt{3},p_{1,1}]$. The second term is positive when both $J_2>0$ and $z<\sqrt{6}$. Because $p_{1,1}\approx 1.84$ for $d=2$, we have $p_{1,1}<\sqrt{6}$ and we are done.

When $d=3$ and $d=4$, we again examine \eqref{dsmall}. Then Lemma~\ref{fact1} together with the argument above give us $j_1^{(4)}>0$ on $(0,\sqrt{d^2-1})\cap(0,p_{1,1}]$. By Proposition~\ref{propLS} we have $p_{1,1}<\sqrt{d+2}$, which for $d=3$ and 4 is less than $\sqrt{d^2=1}$, thus proving the lemma for these $d$.

For dimensions $d\geq 5$, we turn to \eqref{dbigger1}.  The second term is positive on $(0,p_{1,1}]$ for all $d>4$ by Lemma~\ref{fact1}. Since $p_{1,1}<\sqrt{d+2}$ and $\sqrt{d+2}\leq \sqrt{15}$ for $d\leq 13$, we conclude $j_1^{(4)}(z)>0$ on $(0,p_{1,1}]$ for $d=5,\dots,13$.

Finally, suppose $d\geq 14$ and $z\in(0,p_{1,1}]$. If $z\in(0,\sqrt{15}]$, then $j_1^{(4)}(z)>0$ as above. If $z>\sqrt{15}$, then we examine \eqref{dbigger2}. Here the first term is nonnegative on $[\sqrt{15},p_{1,1}]$. The non-Bessel factor of the second term is positive on $(0,\sqrt{\frac{3}{2}(d+6)}]$ and hence on $(0,p_{1,1}]$.
\end{proof}

The next lemma provides some bounds on ultraspherical Bessel functions that will be needed in \cite{chasman} and will not be used in this paper.

Let $d_k$ denote the coefficients of the series expansion for $i_1^{\prime\prime}(z)$, so that
\[
j_1^{\prime\prime}(z) = \sum_{k=1}^\infty (-1)^k d_k z^{2k-1} \quad \text{and} \quad i_1^{\prime\prime}(z) = \sum_{k=1}^\infty d_k z^{2k-1}
\]
by \eqref{jseries} and \eqref{iseries}, where
\begin{align*}
d_k&=\frac{2k+1}{(k-1)!\Gamma(k+1+d/2)}2^{1-2k-d/2}.
\end{align*}

\begin{lemma}\label{ijbounds} We have the following bounds:
 \begin{align*}
 -d_1 z+d_2 z^3&\geq j_1^{\prime\prime}(z)&\text{for all $z\in\Big[0,\sqrt{3(d+2)/(d+5)}\Big]$,}\\
d_1 z+\frac{6}{5}d_2 z^3&\geq i_1^{\prime\prime}(z)&\text{for all $z\in\Big[0,\sqrt{3}\Big]$.}
 \end{align*}
\end{lemma}
\begin{proof}
Let
\[
c_k:=\frac{d_{k+1}}{d_k} = \frac{2k+3}{2k(2k+1)(2k+d+2)}.
\]
It is easy to show that $c_k$ is decreasing for $k \geq 1$.

We use the series expansion to first prove the following upper bound on $j_1^{\prime\prime}(z)$ for $z\geq 0$:
\begin{align*}
(-d_1 z+d_2 z^3)-j_1^{\prime\prime}(z) &= \sum_{k=3}^\infty (-1)^{k+1} d_k z^{2k-1} \\
 &=\sum_{\substack{k=3 \\ k \text{odd}}}^\infty(1-c_k z^2)d_k z^{2k-1}\\
 &\geq (1-c_1 z^2)\sum_{\substack{k=3 \\ k \text{odd}}}^\infty d_k z^{2k-1},
\end{align*}
since $c_k$ is decreasing in $k$. Hence $(-d_1 z+d_2 z^3)-j_1^{\prime\prime}(z)\geq0$ for all $z$ with $0\leq z\leq 1/\sqrt{c_1}=\sqrt{6(d+4)/5}$, which is a larger range even than claimed in the first estimate in the lemma.

For $i_1^{\prime\prime}(z)$ we take a slightly different approach. We will show that on $[0,\sqrt{3}]$,
\[
\frac{1}{5} d_2 z^3 \geq \sum_{k=3}^\infty d_k z^{2k-1},
\]
 and thus
\begin{equation}
 d_1 z+\frac{6}{5}d_2 z^3 \geq i_1^{\prime\prime}(z). \label{ibound}
 \end{equation}
On $[0,\sqrt{3}]$, note that
\begin{align*}
\sum_{k=3}^\infty d_k z^{2k-1} &= \sum_{k=3}^\infty  \frac{2k+1}{(k-1)!\Gamma(k+1+d/2)}2^{-d/2}\left(\frac{z}{2}\right)^{2k-1}\\
&\leq 2^{-d/2}\left(\frac{z}{2}\right)^3\sum_{k=3}^\infty \frac{2k+1}{(k-1)!(k+d/2)\Gamma(k+d/2)}\left(\frac{\sqrt{3}}{2}\right)^{2k-4} \\
&\qquad\text{since $z\leq\sqrt{3}$}
\end{align*}
\begin{align*}
  &\leq \frac{2^{-d/2}}{\Gamma(3+d/2)} \left(\frac{z}{2}\right)^3 \sum_{k=3}^\infty \frac{2k+1}{(k-1)!(k+d/2)}\left(\frac{3}{4}\right)^{k-2}\\
&\qquad\text{since $\Gamma(z)\geq1$ and is increasing on $[2,\infty)$,}\\
  &\leq \frac{d_2}{5}z^3 \sum_{k=2}^\infty \frac{2}{k!} \left(\frac{3}{4}\right)^{k-1}\\
&\qquad\text{by the definition of $d_2$ and taking $k\mapsto k+1$}\\
  &=\frac{8}{15}d_2 z^3\left(e^{3/4}-1-3/4\right)\\
&\qquad\text{by the power series for $e^x$}\\
&\leq \frac{1}{5} d_2 z^3.
\end{align*}

Thus we have obtained our desired bound on $i_1^{\prime\prime}$.\end{proof}

\section{Ultraspherical Bessel functions of the second kind}
Each of the Bessel equations \eqref{besseleqn} and \eqref{modbesseleqn} is a second-order differential equation, and so has another set of solutions --- namely, Bessel and modified Bessel functions of the second kind. However, these functions are singular at the origin. By the regularity of plate eigenfunctions, either these singular solutions do not appear in the eigenfunctions, or they appear in a linear combination such that the singular terms cancel. Lemma~\ref{no2ndkind}, which appears below, states that in fact there is no nontrivial linear combination that meets the smoothness condition.

Ultraspherical Bessel functions of the second kind solve \eqref{besseleqn} and are defined by
\begin{align*}
n_{l}(z) &= z^{-s}N_{s+l}(z)
\end{align*}
with $s=(d-2)/2$. Here $N_\nu(z)$ denotes a Bessel function of the second kind of order $\nu$.  Each $N_\nu(z)$ is linearly independent of $J_\nu(z)$ (see, for example, \cite[p. 358]{AShandbook}), so $n_l(z)$ is linearly independent of $j_l(z)$. The functions $N_l(z)$ are often written as $Y_l(z)$; we use $N_l$ to avoid confusion with the spherical harmonics $\Yl(\thetahat)$.

Ultraspherical modified Bessel functions of the second kind solve \eqref{modbesseleqn} and are defined by
\begin{align*}
k_{l}(z) &= z^{-s}K_{s+l}(z)\\
\text{with}\quad s &= \frac{d-2}{2},
\end{align*}
where $K_\nu(z)$ denotes a modified Bessel function of the second kind of order $\nu$. The $k_l(z)$ are linearly independent of the $i_l(z)$.

\begin{lemma} \label{no2ndkind}
Let $a$, $b$ be positive constants, $a<b$.

For $d\geq 2$ and all integers $l\geq 2$, there is no nontrivial linear combination
\[
 R(z)=An_l(az)+Bk_l(bz)
\]
so that $R(z)$ is smooth at $z=0$.

For $d\geq 2$ and $l=1$, there is no nontrivial linear combination
\[
 R(z)=An_1(az)+Bk_1(bz)
\]
so that $R(z)$ is smooth at $z=0$ with $R(0)=0$.

For $d\geq 2$ and $l=0$, there is no nontrivial linear combination
\[
 R(z)=An_0(az)+Bk_0(bz)
\]
so that $R(z)$ is smooth at $z=0$ with $R'(0)=0$.
\end{lemma}

The proof of the lemma is technical and so has been omitted; it can be found in full in \cite[Chapter 4]{cthesis}.

\section{Proof of Proposition~\ref{eigfc}}
The full set of eigenfunctions for the circular free plate under tension will be found exactly in terms of Bessel and modified Bessel functions and spherical harmonics. 

We will focus on the unit ball, since the solution of our eigenvalue problem for any ball $\BB(R)$ can then be obtained by scaling, as follows. If $u(x)$ is an eigenfunction of the unit ball with eigenvalue $\omega$ and tension $\tau$, then $\tilde{u}(x)=u(x/R)$ is an eigenfunction of the ball $\BB(R)$ with eigenvalue $R^4\omega$ and tension $\tau/R^2$; see \cite[Lemma 24]{chasman}. We will establish Propositions~\ref{eigfc} and~\ref{eigval}, which give us the eigenfunctions and eigenvalues of the ball. Recall that Proposition~\ref{eigfc} states that all eigenfunctions will be of the form $R_l(r)\Yl(\thetahat)$, where $R_l$ is a linear combination (depending on $\tau$) of ultraspherical Bessel and modified Bessel functions of order $l$, and $\Yl$ is a spherical harmonic.

The eigenfunctions of the free disk with zero tension are treated in \cite[Chapter 2]{leissa}, along with treatment of the free disk with clamped and simply supported boundary conditions. In addition to finding the eigenfunctions, Leissa provides some numerical computations of eigenvalues and presents experimental data on such plates.

\subsection*{Spherical harmonics}
In the case where $\Omega$ is the ball, it is natural to consider spherical coordinates. Let $r$ be the radius and $\thetahat$ be the remaining angular information. Consider Laplace's equation $\Delta f = 0$, with $f$ a function on $\RR^d$. The Laplacian can be written in spherical coordinates as
\[
\Delta = \partial_{rr}+\frac{d-1}{r}\partial_r+\frac{1}{r^2}\Delta_S,
\]
where we give the name $\Delta_S$ to the angular part of the Laplacian. Separating variables so that $f=R(r)Y(\thetahat)$, we obtain
\[
 R''+\frac{d-1}{r}R'-\frac{l(l+d-2)}{r^2}R=0 \qquad\text{and}\qquad \Delta_S Y=-l(l+d-2)Y.
\]
The parameter $l$ appearing in the separation constant $l(l+d-2)$ must be an nonnegative integer in order for solutions to exist. The solutions to\\ $\Delta_S Y = -l(l+d-2)Y$ are called the \emph{spherical harmonics}. For each $l$, we choose a spanning set $\{\Yl\}$ of such solutions that are orthonormal in the $L^2(\partial\BB)$ inner product. Because the eigenvalues are real, the $\Yl$ may be chosen to be real-valued. However, they are traditionally chosen to be complex-valued, and so will be treated as possibly such in the proof of Theorem~\ref{thm2}.

\subsection*{Factoring the eigenfunction equation}
The proofs of Propositions~\ref{eigfc} and~\ref{eigval} involve a factoring of the eigenvalue equation akin to factoring of the characteristic equation of an ordinary differential equation.  It is notable that while the eigenvalue equation is not separable in rectangular coordinates, the factoring together with the commutativity of the spherical part of the Laplacian $\Delta_S$ with the Laplacian $\Delta$ gives us a sort of separability in spherical coordinates.

\begin{proof}[Proof of Proposition~\ref{eigfc}]
We first show that eigenfunctions can be written as a product of a radial function with a spherical harmonic, and then give the exact form of the radial part.

Write $A:=\Delta^2-\tau\Delta$. From Section~2, we know each eigenvalue $\omega$ has finite multiplicity, and so the corresponding space of eigenfunctions $X_\omega$ is finite-dimensional. Because $\Delta_S$ is independent of $r$, it commutes with the Laplacian $\Delta$ and hence with our operator $A=\Delta^2-\tau\Delta$. Thus $\Delta_S$ maps $X_\omega$ into itself. The operator $\Ls$ is symmetric, and so is diagonalizable on the finite-dimensional space $X_\omega$. Thus we have that $A$ and $\Ls$ are simultaneously diagonalizable. The eigenfunctions of $\Ls$ on $\partial\BB$ are the spherical harmonics; on $\BB$ the eigenfunctions have the form $R(r)\Yl(\thetahat)$. We can therefore choose our eigenfunctions of $A$ to have this form. 

To find the precise form of $R$, we factor the eigenvalue equation \eqref{maineq}, obtaining
\begin{equation}
(\Delta+a^2)(\Delta-b^2)u = 0, \label{PDEfactored}
\end{equation}
where $a$ and $b$ are positive real numbers satisfying $b^2=a^2+\tau$ and $\omega = a^2(a^2+\tau)$. That is, $a^2=\sqrt{(\tau/2)^2+\omega}-\tau/2$ and $b^2=\sqrt{(\tau/2)^2+\omega}+\tau/2$. The eigenfunctions $u$ will then be linear combinations of the solutions $v_1$ and $v_2$ of each factor:
\begin{equation}\label{factoredpieces}
 (\Delta+a^2)v_1=0 \qquad\text{and}\qquad(\Delta-b^2)v_2=0.
\end{equation}
Each of these is separable in spherical coordinates, with angular equation \\$\Delta_S Y=-l(l+d-2)Y$ for some nonnegative integer $l$. The radial equation for $v_1$ is a rescaling of the ultraspherical Bessel equation \eqref{besseleqn} with order $l$ and the radial equation for $v_2$ is a rescaling of the ultraspherical modified Bessel equation \eqref{modbesseleqn} with order $l$, hence
\[
 v_1=\Big(Aj_{l_1}(ar)+Bn_{l_1}(ar)\Big)Y_{l_1}\qquad\text{and}\qquad v_2=\Big(Ci_{l_2}(br)+Dk_{l_2}(br)\Big)Y_{l_2},
\]
for some nonnegative integers $l_1$, $l_2$ and real constants $A$, $B$, $C$, and $D$.
From the diagonalization argument above, we know $u=R(r)\Yl(\thetahat)$, so all the orders must agree: $l=l_1=l_2$. Thus solutions of the eigenvalue equation \eqref{maineq} have the form
\[
u(x)=R(r)T(\hat{\theta}) = \Big(Aj_l(ar)+Bn_l(ar)+Ci_l(br)+Dk_{l_2}(br)\Big)Y_l(\hat{\theta}).
\]
However, the eigenfunctions are smooth on $\BB$. The spherical harmonics $\Yl$ have no radial dependence; thus we must have the radial part $R(r)$ be a smooth function of $r\in[0,\infty)$. When $l=0$, the spherical harmonic $Y_0$ is constant, and we must also require $R'(0)=0$ in order for $u$ to be smooth. When $l=1$, the spherical harmonics $Y_1$ can be given by $x_i/r$, where $x_i$ are the coordinate functions. Then along the $x_i$-axis, $R(r)Y_1(\thetahat)=R(r)x_i/r=R(r)\sign(x_i)$. This function is continuous only if $R(r)$ vanishes at $r=0$. By Lemma~\ref{no2ndkind}, there is no nontrivial linear combination of Bessel functions of the second kind which satisfies these conditions; thus $B$ and $D$ are both zero. Denote $C/A$ by the constant $\gamma$; then we have
\[
u(x)=R(r)T(\hat{\theta}) = \Big(j_l(ar)+\gamma i_l(br)\Big)Y_l(\hat{\theta}).
\]
The constant $\gamma$ must be chosen so that $u$ satisfies the natural boundary condition $u_{rr}=0$ at $r=1$; hence we have
\[
 \gamma =\frac{-a^2j_l^{\prime\prime}(a)}{b^2i_l^{\prime\prime}(b)},
\]
and so $\gamma$ is real-valued.
\end{proof}

\section{Proof of Proposition~\ref{eigval}}
We prove Proposition~\ref{eigval} for the unit ball; the general result follows from scaling as noted previously. Most of the work has already been done for us in the proof of Proposition~\ref{eigfc}.

\begin{proof}[Proof of Proposition~\ref{eigval}]
It is immediate from the Rayleigh quotient $Q$ that $\omega=0$ is an eigenvalue with corresponding eigenfunction $u=const$.

Fix $\tau>0$. Now suppose $\omega\neq 0$ is an eigenvalue. Then by Proposition~\ref{eigfc} we must have an associated eigenfunction $u$ of the form $\Big(j_l(ar)+\gamma i_l(br)\Big)\Yl$ for some index $l$ and positive real numbers $a$ and $b$ with $b = \sqrt{a^2+\tau}$ and $\omega=a^2b^2$.

All eigenfunctions $u$ satisfy the natural boundary conditions, $Mu=0$ and \\$Vu=0$ on $\partial\Omega$, as given in \eqref{BCb1} and \eqref{BCb2} for the ball. Since by Proposition~\ref{eigfc} all eigenfunctions are linear combinations of $j_l(ar)\Yl(\thetahat)$ and $i_l(br)\Yl(\thetahat)$, we must have some nontrivial linear combination satisfying the homogeneous linear equations
\[
 \begin{cases} Mu=0\\Vu=0.
  \end{cases}
\]
Thus we must have that the determinant
\[
W_l(a):=\det\begin{pmatrix}
Mj_l(ar) & Mi_l(br) \\
Vj_l(ar) & Vi_l(br)\end{pmatrix}_{r=1}
\]
vanishes. From the first natural boundary conditions for the ball given in~\eqref{BCb1}, we have
\begin{equation} \label{Ms}
 Mj_l(ar)=a^2j_l^{\prime\prime}(ar) \qquad\text{and}\qquad Mi_l(ar)=b^2i_l^{\prime\prime}(br).
\end{equation}
The $j_l(ar)$ and $i_l(br)$ are rescaled ultraspherical Bessel and modified Bessel functions, so by the factorization \eqref{factoredpieces}, we have
\[
 \Delta j_l(ar)\Yl(\thetahat)=-a^2j_l(ar)\Yl(\thetahat) \qquad\text{and}\qquad \Delta i_l(br)\Yl(\thetahat)=b^2i_l(br)\Yl(\thetahat).
\]
Then noting $r=1$ on $\partial\BB$, the ``$V$'' boundary condition terms from in \eqref{BCb2} can be rewritten as follows:
\begin{align*}
 Vj_l(a)&= \tau aj_l^\prime(a)+l(l+d-2)\Big(aj_l^\prime(a)-j_l(a)\Big)+a^3j_l'(a)\\
 Vi_l(b)&= \tau bi_l^\prime(b)+l(l+d-2)\Big(bi_l^\prime(b)-i_l(b)\Big)-b^3i_l'(b)\\
\end{align*}

Combining the above with \eqref{Ms} and substituting $\tau=b^2-a^2$, we find
\begin{align}
W_l(a) &= a^2j_l^{\prime\prime}(a)\Big(-a^2bi_l^{\prime}(b)+l(l+d-2)(bi_l^\prime(b)-i_l(b))\Big)\nonumber\\
&\quad-b^2i_l^{\prime\prime}(b)\Big(ab^2j_l^\prime(a)+l(l+d-2)(aj_l^\prime(a)-j_l(a))\Big)\label{MVl}.
\end{align}
Because $\omega$ is an eigenvalue, the determinant $W_l(a)$ vanishes, as desired.

Now suppose we have a number $\omega$ with some integer $l\geq 0$ and positive real numbers $a$, $b$, satisfying $\omega=a^2b^2$, $b^2-a^2=\tau$, and $W_l(a)=0$ for some nonnegative integer $l$. Then the function $u=\Big(j_l(a)+\gamma i_l(b)\Big)\Yl$ satisfies the boundary conditions $Mu=0$ and $Vu=0$ for $r=1$. Furthermore, $u$ satisfies the eigenvalue equation \eqref{maineq} with eigenvalue $a^2b^2=\omega$, and so $\omega$ is an eigenvalue.
\end{proof}

\section{Proof of Theorem~\ref{thm2}}
In this section, we identify the fundamental mode of the ball for positive tension, proving Theorem~\ref{thm2}.

The proof will have two parts. First we show that for any radial function $R(r)$, the Rayleigh quotient $Q[R\Yl]$ is minimized when $l=1$, among all $l\geq1$. Then we show that of all nonconstant eigenstates with $l=0$ and $l=1$, the lowest eigenvalue corresponds to $l=1$. Note that when $l=0$, the spherical harmonic $Y_0$ is the constant function, and so $l=0$ corresponds to purely radial modes.

\begin{proof}[Proof of Theorem 3]
\emph{[Part 1.]} We will show that for any fixed smooth radial function $R$, the Rayleigh quotient $Q[R\Yl]$ is an increasing function in $l$ for all $l\geq 1$. Then by the variational characterization of eigenvalues, we see that the lowest eigenvalue corresponding to an eigenfunction with angular dependence (i.e., $l\geq 1$) occurs when $l=1$.

Considering the numerator and denominator separately, we will use the $L^2(\partial\BB)$-orthonormality of the spherical harmonics to simplify the angular parts of the integrals.

The denominator of our Rayleigh quotient is, for $u=R\Yl$,
\[
\int_\BB |R\Yl|^2\,dx = \int_0^1R^2\,r^{d-1}\,dr,
\]
and so is independent of $l$. So it suffices to show that the numerator is an increasing function of $l$ for $l\geq 1$.

Recall the numerator of the Rayleigh Quotient is
\[
N[u] = \int_\Omega |D^2u|^2+\tau |Du|^2 \, dx.
\]

We use the following pointwise identity to rewrite the Hessian term:
\begin{align}\label{hessstuff}
|D^2u|^2 &= \frac{1}{2}\Big(\Delta(|Du|^2) -D(\Delta u)\cdot D\ubar-D(\Delta\ubar)\cdot Du\Big)
\end{align}
Because our region $\Omega$ is the unit ball, we may use spherical coordinates, noting $\frac{\partial u}{\partial n} = u_r$. We then write the gradient as
\begin{equation}
D u = u_r\rhat + \frac{1}{r}\Ns u,
\end{equation}
where $\rhat = x/r$ is the unit normal. Note that $\frac{1}{r}\Ns$ is the surface gradient $\sgrad$ of the ball, and $\frac{1}{r^2}\Ls$ is the Laplacian on the boundary of the ball. We write $\sdiv$ for the surface divergence of the ball; then $\frac{1}{r^2}\Ls f=\sdiv\sgrad f$.

Then by the Divergence Theorem on $\partial\BB$, we have for any function $f$,
\begin{equation}\label{laptozero}
  \surfintB \frac{1}{r^2}\Ls f\,dS=\surfintB \sdiv\Big(\sgrad f\Big)\,dS=0.
\end{equation}

Exploiting orthonormality of the $\Yl$, we see that
\begin{equation}\label{laprayleigh}
 \surfintB |\Ns\Yl|^2\,dS=l(l+d-2)=:k.
\end{equation}
Thus when $u=R\Yl$, we use \eqref{hessstuff} to rewrite the integral of the Hessian term as follows:
\begin{align*}
\frac{1}{2}&\int_\BB\Big(\Delta(|Du|^2) -D(\Delta u)\cdot D\ubar-D(\Delta\ubar)\cdot Du\Big)\,dx\\
&=\frac{1}{2}\int_\BB\left(\frac{\partial^2}{\partial r^2}+\frac{d-1}{r}\frac{\partial}{\partial r}+\frac{1}{r^2}\Ls\right)\left((R')^2|\Yl|^2+\frac{R^2}{r^2}|\Ns\Yl|^2\right)\\
&\qquad-D\ubar\cdot D\left(u_{rr}+\frac{d-1}{r}u_r-\frac{k}{r^2}u\right)-Du\cdot D\left(\ubar_{rr}+\frac{d-1}{r}\ubar_r-\frac{k}{r^2}\ubar\right)\,dx\\
&\qquad\text{(noting $\Ls\Yl=-k\Yl$)}\\
&=\frac{1}{2}\int_\BB\left(\frac{\partial^2}{\partial r^2}+\frac{d-1}{r}\frac{\partial}{\partial r}\right)\left((R')^2|\Yl|^2+\frac{R^2}{r^2}|\Ns\Yl|^2\right)\\
&\qquad -D\ubar\cdot D\left(u_{rr}+\frac{d-1}{r}u_r-\frac{k}{r^2}u\right)-Du\cdot D\left(\ubar_{rr}+\frac{d-1}{r}\ubar_r-\frac{k}{r^2}\ubar\right)\,dx,
\end{align*}
with this last by noting that \eqref{laptozero} gives us
\[
 \int_\BB\frac{1}{r^2}\Ls\left((R')^2|\Yl|^2+\frac{R^2}{r^2}|\Ns\Yl|^2\right)\,dx=0.
\]
Expanding the integrands, our integral of the Hessian term simplifies to
\begin{align*}
\int_\BB&|\Yl|^2\left((\Rpp)^2+\frac{k+d-1}{r^2}(\Rp)^2-\frac{2k}{r^3}R\Rp\right)\\
&\qquad+|\Ns\Yl|^2\left(\frac{(\Rp)^2}{r^2}-\frac{4}{r^3}R\Rp+\frac{k-d+4}{r^4}R^2\right)\,dS
\end{align*}
Then integrating the above over $\BB$ using \eqref{laprayleigh} and the orthonormality of the $\Yl$, we obtain
\begin{align*}
\int_\BB&|D^2u|^2\,dx\\ &=\int_0^1\left((\Rpp)^2+\frac{2k+d-1}{r^2}(\Rp)^2-\frac{6k}{r^3}R\Rp+\frac{k(k-d+4)}{r^4}R^2\right)r^{d-1}\,dr\\
&=\int_0^1\left((\Rpp)^2+\frac{d-1}{r^2}(\Rp)^2+\frac{2k}{r^4}\left(r\Rp-\frac{3}{2}R\right)^2+\frac{k(k-d-1/2)}{r^4}R^2\right)r^{d-1}\,dr
\end{align*}
with this last equality by completing the square.

We now examine the gradient term in the numerator of the Rayleigh quotient:
\begin{align*}
\int_\BB|Du|^2\,dx &= \int_\BB\left(|u_r|^2+\frac{1}{r^2}|\Ns u|^2\right)\,dx\nonumber\\
&=\int_\BB\left((\Rp)^2|\Yl|^2+\frac{R^2}{r^2}|\Ns\Yl|^2\right)\,dx\nonumber\\
&=\int_0^1\left((\Rp)^2+\frac{k}{r^2}R^2\right)r^{d-1}\,dr,
\end{align*}
again by \eqref{laprayleigh}.

Combining these results, the numerator of the Rayleigh quotient can be now written with all $k$-dependence (and hence $l$-dependence) explicit:
\begin{align*}
N[u]&=\int_0^1\left(\frac{2k}{r^4}\left(r\Rp-\frac{3}{2}R\right)^2+\frac{k(k-d-1/2)}{r^4}R^2+\tau\frac{kR^2}{r^2}\right)r^{d-1}\,dr\\
&\qquad+\int_0^1\left((\Rpp)^2+\frac{d-1}{r^2}(\Rp)^2+\tau(\Rp)^2\right)r^{d-1}\,dr
\end{align*}
Recall $\tau>0$; then the above is increasing with $k$ for $k\geq d+1/2$. Recall\\ $k=l(l+d-2)$ is increasing as a function of $l$; then all terms involving $k$ are increasing functions of $l$ for $l\geq 2$. If $l=1$, the expression $k(k-d-1/2)$ becomes $-3(d-1)/2$, which is negative for all dimensions under consideration. If $l=2$, we find $k(k-d-1/2)= 2d(d-1/2)>0$. Thus each term involving $l$ is increasing as a function of $l$ for all $l\geq 1$.

Thus for any fixed radial function $R$, the numerator $N[R\Yl]$ is an increasing function of $l$ for $l\geq 1$.

\medskip
\emph{[Part 2.]} We now show that the lowest eigenvalue corresponding to an eigenfunction of the form $u_l=\Big(j_l(ar)+\gamma i_l(br)\Big)\Yl$ with $l=1$ is less than the lowest positive eigenvalue for $l=0$.  

Let $\ainf$ denote the first positive zero of $j_1'(a)$. Recall from Proposition~\ref{eigval} that $\omega$ is an eigenvalue if and only if we have some integer $l\geq 0$ and can write $\omega=a^2b^2$ for positive constants $a$ and $b$ such that $b^2-a^2=\tau$ and $W_l(a)=0$. The parameter $\tau$ is positive, so $\omega=a^2(a^2+\tau)$ increases with $a$. Therefore, to show that the lowest nonzero eigenvalue corresponds to $l=1$ and not $l=0$, we show that the first nonzero root of $W_1(a)$ is less than the first nonzero root of $W_0(a)$.

First we consider $l=0$. Here $k = 0$, so we look for solutions to:
\begin{align*}
0&=W_0(a)\\ &=a^2j_0^{\prime\prime}(a)\Big(-a^2bi_0^{\prime}(b)\Big)-b^2i_0^{\prime\prime}(b)\Big(ab^2j_0^\prime(a)\Big)\\
&=a^4bj_1^\prime(a)i_1(b)+ab^4i_1^\prime(b)j_1(a),
\end{align*}
by \eqref{j2} and \eqref{i2}. The functions $i_0(b)$ and $i_0^\prime(b)$ are positive for $b>0$, as noted earlier in this paper, by the power series expansion \eqref{iseries}. Similarly, $j_1(a)$ and $j_1'(a)$ are positive on $(0,\ainf)$ by Lemma~\ref{fact1.5}, and so $W_0(a) > 0$ on $(0,\ainf)$.

Now consider $l=1$. The constant $k=d-1$, so we have
\begin{align}
W_1(a)&=a^2j_1^{\prime\prime}(a)\Big(-a^2bi_1^\prime(b)+(d-1)(bi_1^\prime(b)-i_1(b))\Big)\nonumber\\
&\quad-b^2i_1^{\prime\prime}(b)\Big(ab^2j_1^\prime(a)+(d-1)(aj_1^\prime(a)-j_1(a))\Big)\label{W1a}\\
&=a^2bj_1^{\prime\prime}(a)\Big(-a^2i_1^\prime(b)+(d-1)i_2(b)\Big)\nonumber\\
&\quad-ab^2i_1^{\prime\prime}(b)\Big(b^2j_1^\prime(a)-(d-1)j_2(a)\Big),\label{W1b}
\end{align}
by the Bessel identities \eqref{j2} and \eqref{i2}. As $a\rightarrow 0$, the first term in \eqref{W1b} behaves like
\[
a^2\sqrt{\tau}j_1^{\prime\prime}(a)\Big((d-1)i_2(\sqrt{\tau})\Big)
\]
and so its sign is determined by that of $j_1^{\prime\prime}(a)$. By Lemma~\ref{fact3}, $j_1^{\prime\prime}(a)$ is negative for all $a\in(0,\ainf]$. Hence the first term of $W_1(a)$ is negative near $a=0$.

Similarly, as $a\rightarrow0$, we see $j_2(a)\rightarrow 0$, and so the second term behaves like
\[
-a\tau^2i_1^{\prime\prime}(\sqrt{\tau})j_1^\prime(0),
\]
which is negative by Lemma~\ref{fact1.5}. Therefore $W_1(a) < 0$ near $a=0$.

At $a = \ainf$, we have $j_1^\prime(\ainf) = 0$ by definition of $\ainf$; also note $j_1(\ainf)>0$ and $j_1^{\prime\prime}(\ainf)<0$ by Lemma~\ref{fact1} and~\ref{fact3}. Write $b_\infty=\sqrt{\ainfsq+\tau}$. By Proposition~\ref{propLS}, we have $\ainfsq>d>d-1$ for $d\geq 3$; for $d=2$, we have $\ainfsq \approx 1.84^2>d-1$. Thus $\ainfsq-(d-1)>0$, and so \eqref{W1a} gives
\begin{align*}
W_1(\ainf)&= -\ainfsq j_1^{\prime\prime}(\ainf)\Big(b_\infty(\ainfsq-(d-1))i_1^\prime(b_\infty) +(d-1)i_1(b_\infty)\Big)\\
&\qquad +b_\infty^2i_1^{\prime\prime}(b_\infty)(d-1)j_1(\ainf).
\end{align*}
Hence both terms in $W_1(\ainf)$ and since $W_1$ is continuous, it must have a zero in $(0,\ainf)$. Thus the lowest nonzero root of $W_l(a)$ occurs when $l=1$, not $l=0$, and so the lowest eigenvalue $\omega=a(a^2+\tau)$ occurs when $l=1$. \end{proof}

\section*{Acknowledgments} I am grateful to the University of Illinois
Department of Mathematics and the Research Board for support during my
graduate studies, and the National Science Foundation for graduate student support
under grants DMS-0140481 (Laugesen) and DMS-0803120 (Hundertmark) and DMS 99-83160
(VIGRE), and the University of Illinois Department of Mathematics for travel support to attend the 2007 Sectional meeting of the AMS in New York. I would also like to thank the Mathematisches Forschungsinstitut Oberwolfach for travel support to attend the workshop on Low Eigenvalues of Laplace and Schr\"{o}dinger Operators in 2009. 

Finally, I would like to thank my advisor Richard Laugesen for his support and guidance throughout my time as his student and for his assistance with perparing this paper.

\end{document}